\documentclass[a4paper]{amsart}
\usepackage{amsmath,amssymb,amsxtra,amsfonts,amsthm,comment,url}

\usepackage[margin=3.1cm]{geometry}

\usepackage{tikz}
\usetikzlibrary{arrows,backgrounds,decorations,automata}

\theoremstyle{plain}
\newtheorem{theorem}{Theorem}

\newtheorem{question}[theorem]{Question}

\theoremstyle{definition}

\begin{document}

\title{Extension of a theorem of Duffin and Schaeffer}

\author{Michael Coons}
\address{School of Math.~and Phys.~Sciences\\
University of Newcastle\\
Callaghan\\
Australia}
\email{Michael.Coons@newcastle.edu.au}

\date{\today}

\begin{abstract} Let $r_1,\ldots,r_s:\mathbb{Z}_{n\geqslant 0}\to\mathbb{C}$ be linearly recurrent sequences whose associated eigenvalues have arguments in $\pi\mathbb{Q}$ and let $F(z):=\sum_{n\geqslant 0}f(n)z^n$, where $f(n)\in\{r_1(n),\ldots,$ $r_s(n)\}$ for each $n\geqslant 0$. We prove that if $F(z)$ is bounded in a sector of its disk of convergence, it is a rational function. This extends a very recent result of Tang and Wang, who gave the analogous result when the sequence $f(n)$ takes on values of finitely many polynomials.
\end{abstract}

\maketitle

\section{Introduction}

In 1945, Duffin and Schaeffer \cite{DS1945} proved that {\em if a power series whose coefficients are taken from a finite set is bounded in a sector of the unit disk, then it is already a rational function.} In this short note, we prove the following extension.

\begin{theorem}\label{main} Let $r_1,\ldots,r_s:\mathbb{Z}_{n\geqslant 0}\to\mathbb{C}$ be linearly recurrent sequences whose associated eigenvalues have arguments in $\pi\mathbb{Q}$ and let $F(z):=\sum_{n\geqslant 0}f(n)z^n$, where $f(n)\in\{r_1(n),\ldots,$ $r_s(n)\}$ for each $n\geqslant 0$. If $F(z)$ is bounded in a sector of its disk of convergence, then it is a rational function.
\end{theorem}

Our result, follows from a very recent result of Tang and Wang \cite[Theorem~1.3]{TW2017}, who used a method developed by Borwein, Erd\'elyi, and Littmann \cite{BEL2008} to show that {\em if $p_1(z),\ldots,p_s(z)\in\mathbb{C}[z]$ are polynomials and $F(z):=\sum_{n\geqslant 0}f(n)z^n$, where $f(n)\in\{p_1(n),\ldots,p_s(n)\}$ for each $n\geqslant 0$, is bounded in a sector of the unit disk, then it is a rational function.}

\section{Proof of Theorem \ref{main}}

\begin{proof}[Proof of Theorem \ref{main}] Let $r_1,\ldots,r_s:\mathbb{Z}_{n\geqslant 0}\to\mathbb{C}$ be linearly recurrent sequences whose associated eigenvalues have arguments in $\pi\mathbb{Q}$ and $f(n)\in\{r_1(n),\ldots,r_s(n)\}$ for each $n\geqslant 0$. Then for each $i=1,\ldots,s$, there exist an integer $\ell$, $\ell\times 1$ vectors ${\bf u}_i$ and ${\bf v}_i$ with complex entries, and an $\ell\times\ell$ matrix ${\bf A}$ with complex entries such that $r_i(n)={\bf u}_i^T {\bf A}_i^n {\bf v}_i$. Recall that $r_i(n)$ can be written in the eigenvalue expansion $$r_i(n)=\sum_{j=0}^k \alpha_j n^{m_j-1}\lambda_j^n,$$ where $\lambda_1,\ldots,\lambda_k$ are the eigenvalues of ${\bf A}_i$, which we have assumed have arguments in $\pi\mathbb{Q}$, $m_j$ is the multiplicity of $\lambda_j$, and $\alpha_1,\ldots,\alpha_k\in\mathbb{C}$.

Similarly, we can consider a ``common'' eigenvalue expansion of the linear recurrences $r_1,\ldots,r_s$. In particular, let \begin{equation}\label{rhoorder}\rho_1<\rho_2<\cdots<\rho_D\end{equation} be the distinct moduli of the collection of eigenvalues of ${\bf A}_1,\ldots,{\bf A}_s$. Then there exist integers $D$ and $d$ and a finite collection of complex numbers $\{\alpha_{n,i,m,t}\}_{n\geqslant 0,1\leqslant i\leqslant s,1\leqslant m\leqslant D,0\leqslant d}$, many of which may be zero, such that \begin{equation}\label{eigenr}r_i(n)=\sum_{m=1}^D \sum_{t=0}^d \alpha_{n,i,m,t}n^t\rho_m^n.\end{equation} We note that the collection $\{\alpha_{n,i,m,t}\}_{n\geqslant 0,1\leqslant i\leqslant s,1\leqslant m\leqslant D,0\leqslant d}$ is finite precisely because the all of the eigenvalues associated to the linear recurrences have arguments that are rational multiples of $\pi$---this assumption in the statement of the theorem is crucial for this proof to work.

Now consider $F(z):=\sum_{n\geqslant 0}f(n)z^n$, where $f(n)\in\{r_1(n),\ldots,r_s(n)\}$ for each $n\geqslant 0$. Suppose that $F(z)$ is bounded in a sector of its disk of convergence, that $F(z)$ is irrational, and that $D$ is the  minimal such number so that the previous two properties hold. We will prove the theorem by showing that such a minimal $D$ cannot exist.

To this end, set $\Psi(z):=F(z/\rho_D)$, so that the coefficients $\psi(n)$ of $\Psi(z)$ satisfy $\psi(n)=f(n)\rho_D^{-n}$, $\Psi(z)$ is bounded in a sector of the unit circle, and $\Psi(z)$ is irrational. Using \eqref{eigenr} we have that \begin{equation}\label{psipoly}\psi(n)=\alpha_{n,i_n,D,d}\cdot n^d+\alpha_{n,i_n,D,d-1}\cdot n^{d-1}+\cdots +\alpha_{n,i_n,D,0}+r(n),\end{equation} where for each $n$, \begin{equation}\label{riprime}r(n)\in\left\{r_1'(n),\ldots,r_s'(n)\right\}\quad \mbox{with}\quad r_i'(n)=\sum_{m=1}^{D-1} \sum_{t=0}^d \alpha_{n,i,m,t}n^t(\rho_m/\rho_D)^n,\end{equation} and where here we have used the notation $i_n$ to indicate the value of $i\in\{1,\ldots,s\}$ for which $f(n)=r_i(n)$. By \eqref{rhoorder}, \eqref{riprime}, and the minimality of $D$, we have that the function $R(z)=\sum_{n\geqslant 0} r(n)z^n$ is bounded throughout the unit circle (as it has radius of convergence at least $\rho_D/\rho_{D-1}>1$) and is rational. Thus by \eqref{psipoly}, the function $$\Psi(z)-R(z)=\sum_{n\geqslant 0} \left(\alpha_{n,i_n,D,d}\cdot n^d+\alpha_{n,i_n,D,d-1}\cdot n^{d-1}+\cdots +\alpha_{n,i_n,D,0}\right)z^n$$ is bounded in a sector of the unit circle and irrational. But by Tang and Wang's result, the function $\Psi(z)-R(z)$ is rational. This contradiction proves the theorem.
\end{proof}

\section{A question}

The careful reader will notice that while Tang and Wang's result holds for polynomials in $\mathbb{C}[z]$, the linear recurrences in our theorem, while also being complex-valued, must have associated eigenvalues have arguments in $\pi\mathbb{Q}$. This is because if we let the eigenvalues be more general complex numbers, the arguments of the powers of those eigenvalues may not take only finitely many values. Our proof works as long as this finiteness condition holds.

Because of the necessity of this finiteness condition, Theorem \ref{main} is essentially the best possible result that can be proved from the methods of Borwein, Erd\'elyi, and Littmann \cite{BEL2008}, and Tang and Wang \cite{TW2017}. An interesting generalization would be a statement as in {Theorem~\ref{main}} for general complex linear recurrences. To apply the same basic theory, one would need to answer the following question in the affirmative.

\begin{question} Let $a:\mathbb{Z}_{\geqslant 0}\to \mathbb{C}$ be given by $$a(n)=\sum_{i=0}^M b_{i,n} \zeta_i^n,$$ where $b_{i,n},\zeta_i\in B$, $B$ is finite set of complex numbers, and $\zeta_i$ lies on the unit circle for each $i$. Suppose that for each $\varepsilon>0$ there exist a positive integer $q$ and complex numbers $\gamma_0,\ldots,\gamma_{q-1}$ such that for all $n\geqslant 0$, $$|\gamma_0 a(n)+\cdots+\gamma_{q-1}a(n+q-1)+a(n+q)|<\varepsilon.$$ Is it true that the sequence $a$ is linearly recurrent?
\end{question}

\bibliographystyle{amsplain}
\providecommand{\bysame}{\leavevmode\hbox to3em{\hrulefill}\thinspace}
\providecommand{\MR}{\relax\ifhmode\unskip\space\fi MR }
\providecommand{\MRhref}[2]{%
  \href{http://www.ams.org/mathscinet-getitem?mr=#1}{#2}
}
\providecommand{\href}[2]{#2}


\end{document}